\documentclass{birkjour}
 \newtheorem{thm}{Theorem}[section]
\newtheorem{cor}[thm]{Corollary}

\newtheorem{prop}[thm]{Proposition}
\theoremstyle{definition}
\newtheorem{defn}[thm]{Definition}
\theoremstyle{remark}
\newtheorem{rem}[thm]{Remark}

\numberwithin{equation}{section}
\newcommand{\spn}{{\rm span}}
\begin{document}
	
	\title[On $A$-orthogonality preservation in semi-Hilbert spaces]{On $A$-orthogonality preservation and Blanco-Koldobsky-Turn\v{s}ek theorem in semi-Hilbert spaces}
	

	\author[Manna]{Jayanta Manna}
	\address{Department of Mathematics\\
		 Jadavpur University\\ 
		 Kolkata 700032\\ 
		 West Bengal\\ 
		 INDIA}
	\email{iamjayantamanna1@gmail.com}

	\author[Barik]{Somdatta Barik}
	\address{Department of Mathematics\\ 
		Jadavpur University\\ 
		Kolkata 700032\\ 
		West Bengal\\ 
		INDIA}
	\email{bariksomdatta97@gmail.com}

	\author[Paul] {Kallol Paul}
	\address{Vice-Chancellor, Kalyani University 
		\& 
		Professor of Mathematics\\ 
		Jadavpur University (on lien) \\ 
		Kolkata \\ 
		West Bengal\\ 
		INDIA}
	
	\email{kalloldada@gmail.com}
	
	\author[Sain] {Debmalya Sain}
	\address{Department of Mathematics\\ 
		Indian Institute of Information Technology, Raichur\\ 
		Karnataka 584135 \\
		INDIA}
	\email{saindebmalya@gmail.com}

	
	\begin{abstract}

		 We investigate the local  preservation of $A$-orthogonality at a point by  $A$-bounded operators within the semi-Hilbertian framework induced by a positive operator $A$ on a Hilbert 
space $\mathbb{H}.$ We provide complete characterizations of such preservation. Additionally, we explore properties of the $A$-norm attainment set of an $A$-bounded operator in light of $A$-orthogonality preservation. We also study analogous properties for the minimum $A$-norm attainment set of an $A$-bounded operator. We then characterize the $A$-isometries as the $A$-norm one  operators preserving  $A$-orthogonality. Finally, we characterize those subsets of Hilbert spaces for which such preservation by an $A$-norm one operator implies that the operator is an $A$-isometry.
	\end{abstract}
	\maketitle
	\section{Introduction}
    	The study of orthogonality preservation by bounded linear operators in normed linear spaces plays a significant role because of its close relation with the isometric theory of such spaces. The well-known Blanco-Koldobsky-Turn\v{s}ek theorem \cite{BT06,K93} characterizes the isometries on a normed linear space as the norm one linear operators preserving Birkhoff-James orthogonality. In  \cite{SMP24}, a local study in case of this preservation of  orthogonality by bounded linear operator was introduced, emphasizing its role in analyzing isometries and space geometry. This was extended in \cite{MMPS25} to refine the idea using directional preservation. The purpose of this article is to study the local preservation of $A$-orthogonality  by  $A$-bounded operators within the semi-Hilbertian framework induced by a positive operator $A$ on a Hilbert space and obtain its connection with $A$-isometries. Let us now introduce some relevant notations and terminologies that will be used throughout this article.  \\

        We use $\mathbb{H}$ to refer a Hilbert space. Unless explicitly specified, we consider both real and complex Hilbert spaces. The scalar field is denoted by $\mathbb{K}(\mathbb{R}\text{ or }\mathbb{C}).$ The inner product and the corresponding norm on $\mathbb{H}$ are written as $\langle \cdot, \cdot \rangle$ and $\|\cdot\|$, respectively. For any complex number $\lambda \in \mathbb{C}$, we denote its real part by $\Re(\lambda)$ and its imaginary part by $\Im(\lambda).$ Let $\mathbb{L}(\mathbb{H})$  denote the space of all bounded  linear operators on $\mathbb{H},$ equipped with the standard operator norm. An operator $A \in \mathbb{L}(\mathbb{H})$ is called positive if $\langle Ax, x \rangle \geq 0$ for all $x \in \mathbb{H}$. It is well known that a positive operator $A$ induces a sesquilinear form $\langle \cdot, \cdot \rangle_A$ on $\mathbb{H}$, defined by $\langle x, y \rangle_A = \langle Ax, y \rangle$ for all $x, y \in \mathbb{H}$. This leads to a semi-norm $\|\cdot\|_A$, where $\|x\|_A = \sqrt{\langle Ax, x \rangle}$. From this point onward, we reserve the symbol $A$ for a positive operator on $\mathbb{H}$. The null space and the range of $A$ are denoted by $\mathcal{N}(A)$ and $\mathcal{R}(A)$, respectively. The orthogonal projection on $\overline{\mathcal{R}(A)}$ is denoted by $P_{\overline{\mathcal{R}(A)}}$.

Let the $A$-unit ball and the $A$-unit sphere of  the semi-Hilbert space $\left(\mathbb{H}, \langle \cdot, \cdot\rangle_A\right)$ be denoted by $B_{\mathbb{H}(A)}$ and $S_{\mathbb{H}(A)}$, respectively, i.e.,
\[
B_{\mathbb{H}(A)} = \{x \in \mathbb{H} : \|x\|_A \leq 1\}, \quad S_{\mathbb{H}(A)} = \{x \in \mathbb{H} : \|x\|_A = 1\}.
\]
An operator $T \in \mathbb{L}(\mathbb{H})$ is said to be $A$-bounded if there exists a positive constant $c$ such that $ \|Tx\|_A \leq c \|x\|_A$  for all $x \in \mathbb{H}.$ The set of all such operators is denoted by $B_{A^{1/2}}(\mathbb{H})$, i.e.,
\[
B_{A^{1/2}}(\mathbb{H}) = \{T \in \mathbb{L}(\mathbb{H}) : \text{ there exists } c > 0 \text{ such that } \|Tx\|_A \leq c\|x\|_A \ \forall x \in \mathbb{H} \}.
\]
For any $T \in B_{A^{1/2}}(\mathbb{H})$, the $A$-norm is given by
\[
\|T\|_A = \sup_{\|x\|_A = 1} \|Tx\|_A = \sup \{ |\langle Tx, y \rangle_A| : x, y \in \mathbb{H}, \ \|x\|_A = \|y\|_A = 1 \}.
\]
Correspondingly, the minimum $A$-norm of $T$ is denoted by $m_A(T)$, and is given by
\[
m_A(T) = \inf_{\|x\|_A = 1} \|Tx\|_A = \inf \{ |\langle Tx, y \rangle_A| : x, y \in \mathbb{H}, \ \|x\|_A = \|y\|_A = 1 \}.
\]
Let $T \in B_{A^{1/2}}(\mathbb{H}).$ The $A$-norm attainment set and the minimum $A$-norm attainment set of $T$ are, respectively,
\[
M_T^A = \{ x \in \mathbb{H} : \|x\|_A = 1,\ \|Tx\|_A = \|T\|_A \},\]
\[m_T^A = \{ x \in \mathbb{H} : \|x\|_A = 1,\ \|Tx\|_A = m_A(T) \}.
\]

Let $T \in \mathbb{L}(\mathbb{H}).$ An operator $W \in \mathbb{L}(\mathbb{H})$ is called an $A$-adjoint of $T$ if $\langle Tx,y\rangle_A=\langle x,Wy\rangle_A$ for all $x, y\in \mathbb{H}$. By Douglas theorem \cite{D66}, the set of all operators which admit
an $A$-adjoint is denoted by $B_A(\mathbb{H})$, and is given by
\[B_A(\mathbb{H})=\{T \in \mathbb{L}(\mathbb{H}): \mathcal{R}(T^*A)\subset\mathcal{R}(A)\}.\]
If $T\in B_A(\mathbb{H})$ then the reduced solution of the equation $AX=T^*A$ has a unique solution which is denoted by $T^{\sharp}$ and satisfies $\mathcal{R}(T^\sharp)\subset \overline{\mathcal{R}(A)}$. Note that $T^\sharp=A^\dagger T^*A$, where $A^\dagger$ is the Moore-Penrose inverse of $A,$ see \cite{EN81,P55}. Moreover, $B_A(\mathbb{H})\subset B_{A^{1/2}}(\mathbb{H})\subset\mathbb{L}(\mathbb{H})$. We refer the readers to \cite{ACG09, ACG08, ACG_IEOT_08, SBP24,SSP_AFA_21,SSP_BSM_21,Z19} for further developments in this direction.

        Let us recall the  definition of  $A$-orthogonal from \cite{ACG09}.
       \begin{defn}
 An element $x \in \mathbb{H}$ is $A$-orthogonal to another element $y \in \mathbb{H}$, written as $x \perp_A y$, if the inner product $\langle Ax, y \rangle = 0.$
\end{defn}
When $A$ is the identity operator $I$, this definition aligns with the inner product orthogonality in Hilbert spaces. Observe that, similar to the inner product orthogonality, the $A$-orthogonality relation is homogeneous, symmetric, and both left and right additive. However, for non-zero elements $u, v \in \mathbb{H}$, there may exist multiple scalars $\alpha$ such that $\alpha u + v \perp_A u$. It is important to note that, if $\|u\|_A \neq 0$, then there cannot be more than one such scalar $\alpha$ satisfying $\alpha u + v\perp_A u$.
 It is straightforward to observe that for a positive operator $A$, this orthogonality condition is equivalent to the inequality $\|x + \lambda y\|_A \geq \|x\|_A$  for all scalars $\lambda \in \mathbb{K}$. 
 For  $A=I$,  this corresponds to the notion of Birkhoff-James orthogonality \cite{B35,J47}, i.e., $\|x + \lambda y\| \geq \|x\|$  for all scalars $\lambda \in \mathbb{K}.$ For a comprehensive discussion on Birkhoff-James orthogonality in normed linear spaces, see \cite{MPS24}. Next, we recall the following geometrical notions, introduced in \cite{SSP_BSM_21},  which plays a very crucial role in our exploration.
   Let  $U=\{\alpha\in \mathbb{C}: |\alpha|=1, \arg(\alpha)\in [0,\pi)\}.$ For $x\in \mathbb{H}$ and for $\alpha \in U,$
   \[(x)_{\alpha}^{A+}=\{y\in \mathbb{H}:  \|x + \lambda y\|_A \geq \|x\|_A \text{ for all } \lambda= t\alpha, t\geq0\},\]
   \[(x)_{\alpha}^{A-}=\{y\in \mathbb{H}:  \|x + \lambda y\|_A \geq \|x\|_A \text{ for all } \lambda= t\alpha, t\leq0\},\]
   \[x^{\perp_{\alpha}^A}=\{y\in \mathbb{H}:  \|x + \lambda y\|_A \geq \|x\|_A \text{ for all } \lambda= t\alpha, t\in \mathbb{R}\}.\]
 If $y\in x^{\perp_{\alpha}^A},$ then we write $x\perp_{\alpha}^A y.$ In complex Hilbert spaces, the notions $x^{A+}$, $x^{A-}$ and $x^{\perp_A}$ are defined in the following way:
 \[x^{A+}=\bigcap_{\alpha \in U} (x)_{\alpha}^{A+},~x^{A-}=\bigcap_{\alpha \in U} (x)_{\alpha}^{A-} \text{ and }x^{\perp_A}=\bigcap_{\alpha \in U} x^{\perp_{\alpha}^A}.\]
 For real Hilbert spaces 
 $x^{A+}$, $x^{A-}$ and $x^{\perp_A}$ are defined in the following way:
   \[x^{A+}=\{y\in \mathbb{H}:  \|x + \lambda y\|_A \geq \|x\|_A \text{ for all } \lambda\geq0\},\]
   \[x^{A-}=\{y\in \mathbb{H}:  \|x + \lambda y\|_A \geq \|x\|_A \text{ for all } \lambda\leq0\},\]
   \[x^{\perp_A}=\{y\in \mathbb{H}:  \|x + \lambda y\|_A \geq \|x\|_A \text{ for all } \lambda\in \mathbb{R}\}.\]
   
An operator $T\in B_{A^{1/2}}(\mathbb{H})$ is said to preserve $A$-orthogonality at $x\in \mathbb{H}$ if for any $y\in \mathbb{H},$ $x\perp_A y$ implies that $Tx\perp_A Ty,$ i.e., $\langle x,y\rangle_A=0\implies \langle Tx,Ty\rangle_A=0.$
An operator $ T\in B_{A}(\mathbb{H})$ is called an $A$-isometry if and only if $T^\sharp T=P_{\overline{\mathcal{R}(A)}}$, see \cite{ACG08}. In \cite{SMP24}, the concept of a $\mathcal{K}$-set was introduced for normed linear spaces as a subset of the unit sphere for which the preservation of Birkhoff-James orthogonality by a bounded linear operator at each point implies that the operator is a scalar multiple of an isometry.   In the same spirit, we introduce the following:

 \begin{defn}
 	A set $D\subset S_{\mathbb{H}(A)}$ is said to be a $\mathcal K_A$-set,  if any operator $T\in B_{A}(\mathbb{H})$ that preserves $A$-orthogonality at each point of $D,$  is necessarily a scalar multiple of an $A$-isometry.
 \end{defn}

 \begin{defn}
 	 A set $D\subset S_{\mathbb{H}(A)}$ is said to be a minimal $\mathcal K_A$-set,  if any $\mathcal K_A$-set  $B\subset D$ implies that $B=D.$
 \end{defn}
  The central aim of this article is to study the local preservation of $A$-orthogonality at a point by $A$-bounded operators on a Hilbert space $\mathbb{H}$.  We begin by providing a complete characterization of this local preservation of $A$-orthogonality. Subsequently, we examine various properties of the $A$-norm attainment sets and the minimum $A$-norm attainment sets of $A$-bounded operators, particularly in connection with the preservation of $A$-orthogonality. We also investigate the sets $x^{A+}$ and $x^{A-}$ from the perspective of $A$-orthogonality preservation.  Further characterization of local $A$-orthogonality preservation is obtained in terms of $A$-eigenvectors of the operator $T^\sharp T$. This leads us to identify $A$-isometries as precisely those $A$-norm one operators that preserve $A$-orthogonality. Finally, we provide a complete characterization of $\mathcal{K}_A$-sets in semi-Hilbert spaces.

\section{Main Results}
We begin by establishing a characterization for the directional preservation of $A$-orthogonality at a point by $A$-bounded operators on complex Hilbert spaces. For this purpose, we first obtain the following proposition.

\begin{prop}\label{def alpha}
Let $\mathbb{H}$ be a complex Hilbert space and let $x\in \mathbb{H}$ be such that $\|x\|_A\neq 0.$ Then for any $\alpha\in U,$
\begin{itemize}
    \item[(i)] $(x)_{\alpha}^{A+}=\{y\in \mathbb{H}:  \Re(\alpha\langle y,x \rangle_A)\geq 0\},$
    \item[(ii)] $(x)_{\alpha}^{A-}=\{y\in \mathbb{H}:  \Re(\alpha\langle y,x \rangle_A)\leq 0\},$
    \item[(iii)] $x^{\perp_{\alpha}^A}=\{y\in \mathbb{H}:  \Re(\alpha\langle y,x \rangle_A)= 0\},$
    \item[(iv)] $x^{A+}=\{y\in \mathbb{H}:  \Im(\langle y,x \rangle_A)\leq 0=\Re(\langle y,x \rangle_A)\},$
    \item[(v)] $x^{A-}=\{y\in \mathbb{H}:  \Im(\langle y,x \rangle_A)\geq 0=\Re(\langle y,x \rangle_A)\}.$
\end{itemize}
\end{prop}

\begin{proof}
(i) Let $\alpha \in U$ and let $u\in (x)_{\alpha}^{A+}.$ Then for any $t\geq 0,$ we have
\begin{eqnarray*}
   \|x+t\alpha u\|_A^2\geq \|x\|_A^2
    &\implies& \langle x+t\alpha u, x+t\alpha u\rangle_A \geq \|x\|_A^2\\
    &\implies& 2t \Re(\alpha \langle u, x\rangle_A) + t^2 \|u\|_A^2\geq 0.
\end{eqnarray*}
If $\|u\|_A=0$ then $\Re(\alpha \langle u, x\rangle_A)= 0.$ Suppose that $\|u\|_A\neq0.$
We show that $\Re(\alpha \langle u, x\rangle_A)\geq 0.$ On the contrary suppose that $\Re(\alpha \langle u, x\rangle_A)< 0.$ Consider 
\[t = -\frac{\Re(\alpha \langle u, x\rangle_A)} {\|u\|_A^2}>0.\]
Then we have
\begin{eqnarray*}
&&-2\frac{\Re(\alpha \langle u, x\rangle_A)} {\|u\|_A^2} \Re(\alpha \langle u, x\rangle_A) + \frac{\big(\Re(\alpha \langle u, x\rangle_A)\big)^2} {\|u\|_A^4} \|u\|_A^2\geq 0\\
&\implies&-\frac{\big(\Re(\alpha \langle u, x\rangle_A)\big)^2} {\|u\|_A^2}\geq 0.
\end{eqnarray*}
This is a contradiction. Thus, $\Re(\alpha \langle u, x\rangle_A)\geq 0.$
Therefore, $(x)_{\alpha}^{A+}\subset\{y\in \mathbb{H}:  \Re(\alpha\langle y,x \rangle_A)\geq 0\}.$

Next, let $v\in \{y\in \mathbb{H}:  \Re(\alpha\langle y,x \rangle_A)\geq 0\}.$ Then $\Re(\alpha\langle v,x \rangle_A)\geq 0$ and so for any 
$t\geq 0,$
\begin{eqnarray*}
   \|x+t\alpha v\|_A^2
    &=& \langle x+t\alpha v, x+t\alpha v\rangle_A \\
    &=&  \|x\|_A^2+2t \Re(\alpha \langle v, x\rangle_A) + t^2 \| v\|_A^2\\
    &\geq& \|x\|_A^2.
\end{eqnarray*}
Thus, $v\in (x)_{\alpha}^{A+}$ and consequently, 
$\{y\in \mathbb{H}:  \Re(\alpha\langle y,x \rangle_A)\geq 0\}\subset (x)_{\alpha}^{A+}.$ Therefore,
\[(x)_{\alpha}^{A+}=\{y\in \mathbb{H}:  \Re(\alpha\langle y,x \rangle_A)\geq 0\}.\]

(ii) The proof of (ii) is similar to the proof of (i).

(iii) The proof of (iii) directly follows from  (i) and (ii).

(iv)  Let $u\in \{y\in \mathbb{H}:  \Im(\langle y,x \rangle_A)\leq 0=\Re(\langle y,x \rangle_A)\}.$  Then,
    \begin{eqnarray*}
       && \Im(\langle u,x \rangle_A)\leq 0=\Re(\langle u,x \rangle_A)\\
        &\implies&  \Re(\langle u,x \rangle_A) \cos{\theta} - \Im(\langle u,x \rangle_A) \sin{\theta}\geq 0\text{ for all } 0\leq \theta < \pi\\
        &\implies&  \Re(\alpha \langle u,x \rangle_A) \geq 0 \text{ for all } \alpha \in U.
    \end{eqnarray*}
    From (i), it follows that $u\in (x)_{\alpha}^{A+}$ for all $\alpha \in U.$ Thus, 
    \[u\in \bigcap\limits_{\alpha \in U} (x)_{\alpha}^{A+}=x^{A+}.\]

    Conversely, let $v\in x^{A+}.$ Then it follows from (i)  that 
		$\Re(\alpha\langle v,x \rangle_A)\geq0$ for all $\alpha\in U.$
        This implies that for all $0\leq \theta < \pi,$
        \[ \Re(\langle v,x \rangle_A) \cos{\theta} - \Im(\langle v,x \rangle_A) \sin{\theta}\geq 0.\]
    At $\theta =0$ and $\theta =\frac{\pi}{2},$ we have 
    $\Re(\langle v,x \rangle_A) \geq 0 \text{ and } \Im(\langle v,x \rangle_A) \leq 0, $ respectively.
    Also by taking limit as $\theta\longrightarrow \pi-,$ we have 
    $ \Re(\langle v,x \rangle_A) \leq 0.$
    Thus, $\Im(\langle v,x \rangle_A)\leq 0=\Re(\langle v,x \rangle_A)$ and so $v\in \{y\in \mathbb{H}:  \Im(\langle y,x \rangle_A)\leq 0=\Re(\langle y,x \rangle_A)\}.$
    
(v) The proof of (v) is similar to the proof of (iv).
\end{proof}
 
Next, we have the desired characterization.
\begin{thm}\label{levelal}
    Let $\mathbb{H}$ be a complex  Hilbert space and let $T\in B_{A^{1/2}}(\mathbb{H}).$ Suppose that $x\in \mathbb{H}$ such that $\|x\|_A\neq 0.$  Then  for each $\alpha \in U,$ the following are equivalent:
    \begin{itemize}
        \item[(i)] $T \big(x^{\perp_{\alpha}^A}\big)\subset (Tx)^{\perp_{\alpha}^A}.$
        \item[(ii)] $\Re(\alpha \langle Ty,Tx\rangle_A)= \frac {\|Tx\|_A^2}{\|x\|_A^2} \Re(\alpha \langle y,x\rangle_A) \text{ for all } y\in \mathbb{H}.$
    \end{itemize}
\end{thm}
\begin{proof}
    Sufficient part of the  theorem is obvious. We only prove the necessary part. Let $T \big(x^{\perp_{\alpha}^A}\big)\subset (Tx)^{\perp_{\alpha}^A}.$ Since $\|x\|_A\neq 0,$ it follows that  $\langle Ax,x\rangle\neq 0.$ Now, $\mathbb{H}=\spn~ \{x\}\oplus (Ax)^{\perp}.$ Let $y\in \mathbb{H}.$ Then  $y = \lambda x + v$  for some scalar $\lambda$ and some $v\in (Ax)^{\perp}.$ Hence $\langle v,x\rangle_A=0$ and so $\Re( \alpha\langle v,x\rangle_A)=0.$  Since $T \big(x^{\perp_{\alpha}^A}\big)\subset (Tx)^{\perp_{\alpha}^A},$ it follows that $\Re( \alpha\langle Tx,Tv\rangle_A)=0.$ Now,
    \[\Re(\alpha\langle y,x\rangle_A)=\Re(\alpha\langle \lambda x + v, x\rangle_A)=\|x\|_A^2\Re(\alpha\lambda)  .\]
    Then,
    \begin{eqnarray*}
        \Re(\alpha \langle Ty,Tx\rangle_A)&=&\Re(\alpha\langle \lambda Tx + Tv, Tx \rangle_A)\\
        &=&\Re(\alpha\lambda \|Tx\|_A^2)+\Re(\alpha\langle Tv, Tx\rangle_A)\\
         &=&\|Tx\|_A^2 \Re(\alpha\lambda )+\Re(\alpha\langle Tv, Tx\rangle_A)\\
        &=&\frac {\|Tx\|_A^2}{\|x\|_A^2}\Re(\alpha\langle y,x\rangle_A).
    \end{eqnarray*}
    As $y$ is chosen arbitrarily, the result holds for all $y \in \mathbb{H}$. Hence,
	\[\Re(\alpha \langle Ty,Tx\rangle_A)= \frac {\|Tx\|_A^2}{\|x\|_A^2} \Re(\alpha \langle y,x\rangle_A) \text{ for all } y\in \mathbb{H}\]
    This completes the proof of the theorem.
\end{proof}
We now present a characterization of the local preservation of  $A$-orthogonality at a point by $A$-bounded operators on arbitrary Hilbert spaces.
\begin{thm}\label{level}
    Let $T\in B_{A^{1/2}}(\mathbb{H}).$ Suppose that $x\in \mathbb{H}$ such that $\|x\|_A\neq 0.$  Then  $T$ preserves $A$-orthogonality at $x$ if and only if  
			 \[\langle Ty, Tx \rangle_A= \frac {\|Tx\|_A^2}{\|x\|_A^2} \langle y,x \rangle_A \text{ for all } y\in \mathbb{H}.\]
\end{thm}
\begin{proof}
    We only prove the necessary part as the sufficient part  is obvious. Suppose  $T$ preserves $A$-orthogonality at $x.$ Now, $\mathbb{H}$ can be expressed as  $\mathbb{H}=\spn~ \{x\}\oplus (Ax)^{\perp}.$ Let $y\in \mathbb{H}.$  Then there exist $\alpha\in \mathbb{K}$ and $v\in (Ax)^{\perp}$ such that   $y = \alpha x + v.$  Hence $\langle v,x\rangle_A=0$ and so $\langle Tv,Tx\rangle_A=0.$   Now,
	\[\langle y,x\rangle_A=\langle \alpha x + v, x\rangle_A=\alpha \|x\|_A^2 .\]
	Then,
	\[\langle Ty,Tx\rangle_A=\langle \alpha  Tx + Tv, Tx \rangle_A
		=\alpha \|Tx\|_A^2
	=\frac {\|Tx\|_A^2}{\|x\|_A^2}\langle y,x\rangle_A.\]
Since $y\in \mathbb{H}$ is chosen arbitrarily, it follows that 
\[\langle Ty, Tx \rangle_A= \frac {\|Tx\|_A^2}{\|x\|_A^2} \langle y,x \rangle_A \text{ for all } y\in \mathbb{H}.\]
	This completes the proof.
\end{proof}
By using the above result, we derive the following corollary. To proceed, we recall from \cite{ACG08} that for  any $T\in B_{A^{1/2}}(\mathbb{H}),$ the set $\mathcal{N}_A(T)$ defined as $\{u\in \mathbb{H}: \|Tu\|_A=0\}.$
\begin{cor}
     Let $T\in B_{A^{1/2}}(\mathbb{H})$ and let  $x\in \mathbb{H}.$ If any $T$ preserves $A$-orthogonality at $x$ then either $\|Tx\|_A=0$ or $\mathcal{N}_A(T)\subset x^{\perp_A}.$ 
\end{cor}
\begin{proof} Let  $T$ preserve $A$-orthogonality at $x.$ Suppose that $\|Tx\|_A\neq0.$
    If $\|x\|_A=0$ then it is easy to observe that $\|Tx\|_A=0$ and so  $\|x\|_A\neq0.$  Let $y\in \mathcal{N}_A(T).$ Then $\langle Ty, Tx\rangle_A=0.$ From Theorem \ref{level}, it follows that $\langle y, x\rangle_A=0.$ Thus,  $y\in x^{\perp_A}$ and consequently, $\mathcal{N}_A(T)\subset x^{\perp_A}.$
\end{proof}

Now, we have an observation that the preservation of $A$-orthogonality at any point by an  $A$-bounded operator $T$ on a Hilbert space implies that either this point belongs to $m_T^A$ or $A$-orthogonal to $m_T^A.$
\begin{prop}\label{prop mT sub}
 Let  $T\in B_{A^{1/2}}(\mathbb{H}).$ If  $T$ preserves $A$-orthogonality at any $x\in S_{\mathbb{H}(A)}\setminus m_T^A$ then  $m_T^A\subset x^{\perp_A}.$ 
\end{prop}
\begin{proof}
    Let $z\in m_T^A$ and let  $T$ preserve $A$-orthogonality   at $x\in S_{\mathbb{H}(A)}\setminus m_T^A.$ Then clearly $\|Tx\|_A>m_A(T).$  Now, $\mathbb{H}=\spn~\{x\}\oplus (Ax)^\perp$ and so  $z=\alpha x+v$ for some $\alpha \in \mathbb{K}$ and $v\in (Ax)^\perp.$  If possible suppose that $\alpha\neq 0.$ Since  $T$ preserves $A$-orthogonality at $x,$ it follows that $\langle Tx,Tv\rangle_A=0.$ Then 
    \begin{eqnarray*}
        \|Tz\|_A^2&=&\langle T(\alpha x+v),T(\alpha x+v)\rangle_A\\
        &=& |\alpha|^2 \|Tx\|_A^2+\|Tv\|_A^2\\
         &=& |\alpha|^2 \|Tx\|_A^2+\|v\|_A^2 \left\|T\left(\frac{v}{\|v\|_A}\right)\right\|_A^2\\
        &>& |\alpha|^2 \big(m_A(T)\big)^2+ \|v\|_A^2 \big(m_A(T)\big)^2\\
        &=& \big(|\alpha|^2+\|v\|_A^2\big) \big(m_A(T)\big)^2.
    \end{eqnarray*}
    Since $\langle x,v\rangle_A=0,$  we have $|\alpha|^2+\|v\|_A^2=\|\alpha x+v\|_A=\|z\|_A=1$ and so $\|Tz\|_A> m_A(T).$ This is a contradiction. Thus, $\alpha=0$ and so $z\in x^{\perp_A}.$ Therefore,  $m_T^A\subset x^{\perp_A}.$
\end{proof}

Next, we provide complete characterizations of the $A$-norm attainment set  and  the minimum  $A$-norm attainment set for an $A$-bounded linear operator  on a Hilbert space.
\begin{thm}\label{level MT}
    	Let $T\in B_{A^{1/2}}(\mathbb{H}).$  For any $x\in S_{\mathbb{H}(A)},$ the following results hold.
        \begin{itemize}
            \item[(i)] $x\in M_T^A\iff\langle Ty,Tx\rangle_A= \|T\|_A^2 \langle y,x \rangle_A \text{ for all } y\in \mathbb{H}.$
            \item[(ii)]$x\in m_T^A\iff\langle Ty,Tx\rangle_A= (m_A(T))^2 \langle y,x\rangle_A \text{ for all } y\in \mathbb{H}.$
        \end{itemize}
\end{thm}
\begin{proof}
    (i) Let $x\in M_T^A.$ First, we show  that $x \perp_A z\implies Tx \perp_A Tz$ for any $z\in \mathbb{H}.$ Suppose on the contrary that there exists a non-zero $z\in \mathbb{H}$ be such that $\langle z,x\rangle_A=0$ but  $\langle Tz,Tx\rangle_A\neq0.$  Then for any $\lambda \in \mathbb{K},$
	\begin{eqnarray*}
\|Tx\|_A^2(1+|\lambda|^2\|z\|_A^2)&=&\|T\|_A^2\|x+\lambda z\|_A^2\\
		&\geq& \|Tx+\lambda Tz\|_A^2\\
        &=& \|Tx\|_A^2+ 2\Re (\lambda \langle Tz,Tx\rangle_A )+ |\lambda|^2\|Tz\|_A^2.
        \end{eqnarray*}
		This implies that
        \[|\lambda|^2\|z\|_A^2 \|Tx\|_A^2\geq 2\Re (\lambda \langle Tz,Tx\rangle_A  )+ |\lambda|^2\|Tz\|_A^2.\]

	Consider $\lambda = \frac{\epsilon}{\langle Tz,Tx\rangle_A},$ where $0<\epsilon <\frac{2|\langle Tz,Tx\rangle_A|^2}{\|z\|_A^2 \|Tx\|_A^2}.$ Then 
\[\frac{\epsilon^2 \|z\|_A^2 \|Tx\|_A^2}{|\langle Tz,Tx\rangle_A|^2}> 2\epsilon+ |\lambda|^2\|Tz\|_A^2\\
		 \implies 2\epsilon > 2\epsilon+ |\lambda|^2\|Tz\|_A^2.
\]
This is a contradiction. Thus, $x \perp_A z\implies Tx \perp_A Tz$ for any $z\in \mathbb{H}.$  Then it follows from Theorem \ref{level} that 
\[\langle Tz,Tx\rangle_A= \|T\|_A^2 \langle z,x\rangle_A \text{ for all } z\in \mathbb{H}.\]

Conversely,  let $\langle Ty,Tx\rangle_A= \|T\|_A^2 \langle y,x\rangle_A \text{ for all } y\in \mathbb{H}.$ Then $\langle Tx,Tx\rangle_A= \|T\|_A^2 \langle x,x\rangle_A$ and so $\|Tx\|_A= \|T\|_A.$ Thus, $x\in M_T^A.$ This completes the proof of (i).\\

(ii) Let $x\in m_T^A$  and let  $z\in \mathbb{H}$ be such that $x\perp_A z.$ Then for any $\lambda \in \mathbb{K},$
\[\|x+\lambda z\|_A\geq \|x\|_A=1.\]
Now, 
\begin{eqnarray*}
    \|Tx+\lambda Tz\|_A &=&\|T(x+\lambda z)\|_A\\
    &\geq& m_A(T) \|x+\lambda z\|_A\\
    &=& \|Tx\|_A\|x+\lambda z\|_A\\
    &\geq&\|Tx\|_A.
\end{eqnarray*}
So, $Tx\perp_A Tz.$ 
Then it follows from Theorem \ref{level} that 
\[\langle Tz,Tx\rangle_A= (m_A(T))^2 \langle z,x\rangle_A \text{ for all } z\in \mathbb{H}.\]

Conversely,  let $\langle Ty,Tx\rangle_A= (m_A(T))^2 \langle y,x\rangle_A \text{ for all } y\in \mathbb{H}.$ Then $\langle Tx,Tx\rangle_A= (m_A(T))^2 \langle x,x\rangle_A$ and so $\|Tx\|_A= m_A(T).$ Thus, $x\in m_T^A.$ This completes the proof of (ii).
\end{proof}
\begin{rem}
\begin{itemize}
    \item[(i)] From Theorem \ref{level MT}, it is straight forward to see that any $A$-bounded operator on a Hilbert space  preserves $A$-orthogonality in both direction at each element of the $A$-norm attainment set and the minimum  $A$-norm attainment set of that operator, i.e., for any $T\in B_{A^{1/2}}(\mathbb{H}),$ if $x\in  M_T^A(~\text{or}~  m_T^A)$ then $x\perp_Ay\iff Tx\perp_ATy. $ 
    \item[(ii)] Since any $A$-bounded operator on a Hilbert space  preserves $A$-orthogonality  at each element of $A$-norm attainment set, it follows from Proposition \ref{prop mT sub} that for any $A$-bounded operator, the  $A$-norm attainment set and  the minimum $A$-norm attainment set are either equal or $A$-orthogonal. 
\end{itemize}
\end{rem}

The next two corollaries present further significant properties of the $A$-norm attainment set and the minimum $A$-norm attainment set of an $A$-bounded operator on a Hilbert space. We provide the proof of the first corollary only, as the second can be proven in a similar manner.
\begin{cor}
      Let  $T\in B_{A^{1/2}}(\mathbb{H})$ be  such that $\|T\|_A\neq 0.$ Then for any  $x\in M_T^A$, 
    \begin{itemize}
        \item[(i)]  $ T \big(x^{A+}\big)= (Tx)^{A+},$
        \item[(ii)] $ T \big(x^{A-}\big)= (Tx)^{A-},$ 
        \item[(iii)] $T \big(x^{\perp_{\alpha}^A}\big)= (Tx)^{\perp_{\alpha}^A}.$
    \end{itemize}
\end{cor}
\begin{proof}
  (i) We first consider $\mathbb{H}$ to be a complex Hilbert space. Let  $\alpha \in U.$ 
  From Theorem \ref{level MT}, it follows that  
  \begin{eqnarray*}
     && \langle Ty,Tx\rangle_A= \|T\|_A^2 \langle y,x\rangle_A \text{ for all } y\in \mathbb{H}\\
     &\implies&  \Re(\alpha\langle Ty,Tx\rangle_A)= \|T\|_A^2 \Re(\alpha\langle y,x\rangle_A) \text{ for all } y\in \mathbb{H}.
  \end{eqnarray*}
  This implies that 
  \[\Re(\alpha\langle Ty,Tx\rangle_A)\geq 0 \iff \Re(\alpha\langle y,x\rangle_A)\geq 0. \]
  Therefore, $T \big((x)_{\alpha}^{A+}\big)= (Tx)_{\alpha}^{A+}.$ This holds for each $\alpha \in U$ and so $T \big(x^{A+}\big)= (Tx)^{A+}.$\\
  For real Hilbert spaces, the result follows by noting that $x^{A+}=(x)_{\alpha}^{A+},$ for $\alpha=1.$
  The proofs of (ii) and (iii) follow from similar arguments.
\end{proof}

\begin{cor}
      Let $T\in B_{A^{1/2}}(\mathbb{H})$ be  such that $m_A(T)\neq 0.$ Then for any  $x\in m_T^A$, 
    \begin{itemize}
        \item[(i)]  $ T \big(x^{A+}\big)= (Tx)^{A+},$
        \item[(ii)] $ T \big(x^{A-}\big)= (Tx)^{A-},$ 
        \item[(iii)] $T \big(x^{\perp_{\alpha}^A}\big)= (Tx)^{\perp_{\alpha}^A}.$
    \end{itemize}
\end{cor}
Next, we establish a distinct characterization of the local preservation of $A$-orthogonality  at a point by $A$-bounded operators acting on complex Hilbert spaces.
\begin{thm}
    Let $\mathbb{H}$ be a complex Hilbert space and let  $x\in \mathbb{H}$ be such that $\|x\|_A\neq0.$ Suppose that  $T\in B_{A^{1/2}}(\mathbb{H})$ and $\alpha\in U.$ Consider  the following statements.
    \begin{itemize}
        \item[(i)] $T \big(x^{\perp_{\alpha}^A}\big)\subset (Tx)^{\perp_{\alpha}^A}.$
\item[(ii)] Either $ T \big((x)_{\alpha}^{A+}\big)\subset (Tx)_{\alpha}^{A+}$ and $ T \big((x)_{\alpha}^{A-}\big)\subset (Tx)_{\alpha}^{A-},$ or $ T \big((x)_{\alpha}^{A+}\big)\subset (Tx)_{\alpha}^{A-}$ and $ T \big((x)_{\alpha}^{A-}\big)\subset (Tx)_{\alpha}^{A+}.$ 

        \item[(iii)] $T \big(x^{\perp_A}\big)\subset (Tx)^{\perp_A}.$

        \item[(iv)] Either $ T \big(x^{A+}\big)\subset (Tx)^{A+}$ and $ T \big(x^{A-}\big)\subset (Tx)^{A-},$ or $ T \big(x^{A+}\big)\subset (Tx)^{A-}$ and $ T \big(x^{A-}\big)\subset (Tx)^{A+}.$ 
    \end{itemize}
    Then (i) is equivalent to (ii),   and (iii) is equivalent to (iv).
\end{thm}

\begin{proof}
We only prove the equivalence of (i) and (ii). The equivalence between (iii) and (iv) can be done in a similar way.

       (i)$\implies$(ii):  Let $T \big(x^{\perp_{\alpha}^A}\big)\subset (Tx)^{\perp_{\alpha}^A}.$ If $\|Tx\|_A=0$ then (ii) follows trivially. Let $\|Tx\|_A\neq0.$ If possible suppose that  there exist $u,v \in (x)_{\alpha}^{A+}\setminus x^{\perp_{\alpha}^A}$ such that  
          \[Tu\in (Tx)_{\alpha}^{A+}\setminus (Tx)^{\perp_{\alpha}^A} \text{ and } Tv\in (Tx)_{\alpha}^{A-}\setminus (Tx)^{\perp_{\alpha}^A}.\] 
      Then from Proposition \ref{def alpha}, it follows that 
      \[\Re(\alpha\langle Tu,Tx \rangle_A)>0 \text{ and } \Re(\alpha\langle Tv,Tx \rangle_A)<0.\]
      Then there exists  $t>0$ such that 
      \begin{eqnarray*}
          &&(1-t)\Re(\alpha\langle Tu,Tx \rangle_A)+ t \Re(\alpha\langle Tv,Tx \rangle_A)=0\\
          &\implies& \Re(\alpha\langle (1-t)Tu +tTv,Tx \rangle_A)=0\\
          &\implies& \Re(\alpha\langle T((1-t)u +tv),Tx \rangle_A)=0.
      \end{eqnarray*}
     From  Theorem \ref{levelal}, it follows that 
     \[\frac {\|Tx\|_A^2}{\|x\|_A^2}\Re(\alpha\langle (1-t)u +tv),x\rangle_A)=\Re(\alpha\langle T((1-t)u +tv),Tx \rangle_A)=0\]
     This implies that 
     \begin{eqnarray*}
        && \Re(\alpha\langle (1-t)u +tv, x\rangle_A)=0\\
         &\implies& (1-t)\Re(\alpha\langle u,x\rangle_A) +t\Re(\alpha\langle v,x\rangle_A)=0\\
         &\implies& \Re(\alpha\langle u,x\rangle_A) \Re(\alpha\langle v,x\rangle_A)\leq 0.
     \end{eqnarray*}
   This contradicts the fact that $u,v \in (x)_{\alpha}^{A+}\setminus x^{\perp_{\alpha}^A}.$ Thus, \[\text{ either }   T \big((x)_{\alpha}^{A+}\big)\subset (Tx)_{\alpha}^{A+}\text{ or }T \big((x)_{\alpha}^{A+}\big)\subset (Tx)_{\alpha}^{A-}.\] 
Now, it is easy to observe that 
\begin{eqnarray*}
   && T \big((x)_{\alpha}^{A+}\big)\subset (Tx)_{\alpha}^{A+}\iff T \big((x)_{\alpha}^{A-}\big)\subset (Tx)_{\alpha}^{A-}\\
   &\text{and}& T \big((x)_{\alpha}^{A+}\big)\subset (Tx)_{\alpha}^{A-}\iff T \big((x)_{\alpha}^{A-}\big)\subset (Tx)_{\alpha}^{A+}.
\end{eqnarray*}
 Therefore,  either $ T \big((x)_{\alpha}^{A+}\big)\subset (Tx)_{\alpha}^{A+}$ and $ T \big((x)_{\alpha}^{A-}\big)\subset (Tx)_{\alpha}^{A-},$ or $ T \big((x)_{\alpha}^{A+}\big)\subset (Tx)_{\alpha}^{A-}$ and $ T \big((x)_{\alpha}^{A-}\big)\subset (Tx)_{\alpha}^{A+}.$ \\
 
   (ii)$\implies$(i): Let $ T \big((x)_{\alpha}^{A+}\big)\subset (Tx)_{\alpha}^{A+}$ and  $ T \big((x)_{\alpha}^{A-}\big)\subset (Tx)_{\alpha}^{A-}.$ Then 
    \[  T \big((x)_{\alpha}^{A+}\big)\cap  T\big((x)_{\alpha}^{A-}\big)\subset (Tx)_{\alpha}^{A+} \cap (Tx)_{\alpha}^{A-}
       \implies T \big(x^{\perp_{\alpha}^A}\big)\subset (Tx)^{\perp_{\alpha}^A}.\]
Similarly, if $ T \big((x)_{\alpha}^{A+}\big)\subset (Tx)_{\alpha}^{A-}$ and  $ T \big((x)_{\alpha}^{A-}\big)\subset (Tx)_{\alpha}^{A+},$ then also $T \big(x^{\perp_{\alpha}^A}\big)\subset (Tx)^{\perp_{\alpha}^A}.$

\end{proof}
\begin{rem}
By considering $\alpha=1$  and noting  that $x^{A+}=(x)_{\alpha}^{A+}$ and  $x^{A-}=(x)_{\alpha}^{A-},$  it is easy to see that the  above result holds for real Hilbert space also. 
\end{rem}

We now shift our focus to explore the relationship between $A$-orthogonality preservation by $A$-bounded operators and  $A$-isometries on a Hilbert space. To initiate this investigation, we present the following result, which provides a characterization of $A$-eigenvectors of the operator $T^\sharp T$ for any $T \in B_{A}(\mathbb{H})$. Recall from \cite{SBP24} that a non-zero element $x \in \overline{\mathcal{R}(A)}$ is said to be an $A$-eigenvector of $T$ if there exists $\lambda \in \mathbb{K}$ such that $ATx = \lambda Ax$, where $\lambda$ is referred to as the corresponding $A$-eigenvalue of $T.$ The set of all $A$-eigenvectors associated with a particular $A$-eigenvalue of $T$ along with zero is called the $A$-eigenspace.

	\begin{prop}\label{evt}
		 Let  $T\in B_{A}(\mathbb{H}).$ Suppose that $x\in\overline{\mathcal{R}(A)}$ is non-zero.  Then $T$ preserves  $A$-orthogonality at $x$ if and only if $x$ is an $A$-eigenvector of $T^\sharp T.$ 
	\end{prop}
	\begin{proof}
		First, we prove the sufficient part. Let $x\in\overline{\mathcal{R}(A)}$ be an $A$-eigenvector of $T^\sharp T.$ Then $AT^\sharp Tx=\lambda A x$ for some $\lambda \in\mathbb{K}.$ Now, for any $y\in\mathbb{H},$ $\langle Tx,Ty\rangle_A=\langle T^\sharp Tx,y\rangle_A=\langle AT^\sharp Tx,y\rangle=\langle \lambda A x,y\rangle=\lambda \langle x,y\rangle_A.$ Thus,  for any $y\in\mathbb{H},$ $\langle x,y\rangle_A=0\implies\langle Tx,Ty\rangle_A=0.$ Therefore, $T$ preserves $A$-orthogonality at $x.$
		
		Next, we prove the necessary part. Suppose that $T$ preserves $A$-orthogonality at $x.$  Without loss of generality suppose that $\|x\|_A=1.$ Now,  $\mathbb H=\spn~ \{x\}\oplus (Ax)^{\perp}$.  Let $y\in\mathbb{H}.$ Then $y = \alpha x + v$  for some scalar $\alpha$ and some $v\in (Ax)^{\perp}.$  Since  $T$ preserves $A$-orthogonality at $x$ and $\langle x, v\rangle_A=0,$ it follows that $\langle Tx, Tv\rangle_A =0.$
		Now,
		\begin{eqnarray*}
			\langle Ty, Tx\rangle_A& =&\langle \alpha Tx + Tv, Tx\rangle_A\\
			&=&\alpha \|Tx\|_A^2\\
			&=&\|Tx\|_A^2\langle \alpha x+v, x\rangle_A\\
			&=& \|Tx\|_A^2 \langle y,x\rangle_A.
		\end{eqnarray*}
	Thus, for all $y\in \mathbb{H},$ we have
	\begin{eqnarray*}
		&&\langle Tx, Ty\rangle_A=\langle\|Tx\|_A^2  x,y\rangle_A\\
		\implies && \langle T^\sharp Tx- \|Tx\|_A^2x, y\rangle_A =0\\
		\implies && \langle AT^\sharp Tx- \|Tx\|_A^2Ax, y\rangle =0.
	\end{eqnarray*}
	So
	 $AT^\sharp Tx- \|Tx\|_A^2Ax=0,$ i.e., $AT^\sharp Tx=\|Tx\|_A^2Ax.$ Therefore, $x$ is an $A$-eigenvector of $T^\sharp T$ corresponding to the $A$-eigenvalue $\|Tx\|_A^2.$ This completes the proof of the proposition.
	\end{proof}

Next, we provide a complete characterization of $A$-isometries within the class of $A$-bounded operators that admit an $A$-adjoint.
\begin{thm}\label{iso}
	Let $T\in B_{A}(\mathbb{H}).$ Then $T$ preserves  $A$-orthogonality at each $x\in\overline{\mathcal{R}(A)}$ if and only if $T$ is a scalar multiple of an $A$-isometry. 
\end{thm}
\begin{proof}
	Sufficient part of the theorem is obvious. We only prove the necessary part. Let $T$ preserve  $A$-orthogonality at each $x\in\overline{\mathcal{R}(A)}.$ From Proposition \ref{evt}, it follows that every  $x\in\overline{\mathcal{R}(A)}$ is an $A$-eigenvector of $T^\sharp T$ with  $A$-eigenvalue $\|Tx\|_A^2.$ So 
	\[AT^\sharp Tx=\|Tx\|_A^2Ax, \text{ for all } x\in \overline{\mathcal{R}(A)}.\]
	We claim that $ \|Tx\|_A=\|Ty\|_A$ for all $x, y \in \overline{\mathcal{R}(A)}.$ If $x,y$ are linearly dependent then this is obvious. Let $x,y$ be linearly independent. Then $Ax, Ay$ are also linearly independent. Now, $x+y$ is  also an $A$-eigenvector of $T^\sharp T$ with  $A$-eigenvalue $\|Tx+Ty\|_A^2.$ So
	\begin{eqnarray*}
		&&AT^\sharp T(x+y)=\|Tx+Ty\|_A^2A(x+y)\\
		\implies&& AT^\sharp Tx+AT^\sharp Ty=\|Tx+Ty\|_A^2Ax+\|Tx+Ty\|_A^2Ay\\
		\implies&& \|Tx\|_A^2Ax+\|Ty\|_A^2Ay=\|Tx+Ty\|_A^2Ax+\|Tx+Ty\|_A^2Ay\\
			\implies&& \left(\|Tx\|_A^2-\|Tx+Ty\|_A^2\right)Ax+\left(\|Ty\|_A^2-\|Tx+Ty\|_A^2\right)Ay=0.
	\end{eqnarray*}
	Since $Ax, Ay$ are linearly independent, it follows that $\|Tx\|_A=\|Tx+Ty\|_A=\|Ty\|_A=k (\text{say}).$ Thus,  our claim is established. Therefore, $AT^\sharp Tx=k^2 Ax$ for all $x\in \overline{\mathcal{R}(A)}$ and so $AT^\sharp Tx=k^2 Ax$ for all $x\in \mathbb{H}.$ So
	\begin{eqnarray*}
	AT^\sharp T=k^2 A
		&\implies&AA^\dagger T^*A T=k^2 A\\
		&\implies& A^\dagger AA^\dagger T^*A T=k^2 A^\dagger A\\
			&\implies& A^\dagger T^*A T=k^2 P_{\overline{\mathcal{R}(A)}}\\
				&\implies& T^\sharp T=k^2 P_{\overline{\mathcal{R}(A)}}.
	\end{eqnarray*}
	Thus, $T$ is an $A$-isometry multiplied by $k.$
\end{proof}

Finally, we obtain a complete description of   $\mathcal K_A$-sets in Hilbert spaces. For this, we require the following observaion on the size of these sets in  Hilbert spaces. 
 \begin{prop}\label{dokn}
	  If $D\subset \overline{\mathcal{R}(A)}$ is a $\mathcal K_A$-set then $\spn~ D=\overline{\mathcal{R}(A)}.$ 
\end{prop}
\begin{proof}
	Suppose on the contrary that $D\subset \overline{\mathcal{R}(A)}$ is a $\mathcal K_A$-set  such that  with $\overline{\mathcal{R}(A)} \setminus \spn~ D\neq \emptyset.$ Clearly, $(\overline{\mathcal{R}(A)}, \langle\cdot,\cdot\rangle_A )$ forms a Hilbert space and so $\overline{\mathcal{R}(A)}= (\spn~ D) \oplus \Big((\spn~ D)^{\perp_A}\cap \overline{\mathcal{R}(A)} \Big).$ Then  $$\mathbb{H}=\overline{\mathcal{R}(A)}\oplus \mathcal{N}(A)=(\spn~ D) \oplus \Big((\spn~ D)^{\perp_A}\cap \overline{\mathcal{R}(A)} \Big)\oplus \mathcal{N}(A).$$
	 Consider the operator $T \in B_{A}(\mathbb{H})$ defined by  \[Tu=\begin{cases}
	 	0,&~\text{ if } u\in \spn~ D\\
	 	u,&~\text{ if } u\in \Big((\spn~ D)^{\perp_A}\cap \overline{\mathcal{R}(A)} \Big)\\
	 	0,&~\text{ if } u\in \mathcal{N}(A).
	 \end{cases}\] 
	  Then it is clear that $T$ preserves $A$-orthogonality at $D$ but $T$ is not a scalar multiple of an $A$-isometry. This contradicts the hypothesis that $D$ is a $\mathcal K_A$-set. Therefore, $\spn~ D=\overline{\mathcal{R}(A)}.$
\end{proof}
We are now ready to prove the desired result.
	\begin{thm}\label{k_set}
		 A set $D\subset \overline{\mathcal{R}(A)}$ is a $\mathcal K_A$-set if and only if $D$  satisfies the  following two conditions:
		\begin{itemize}
			\item[(i)] $\spn~ D=\overline{\mathcal{R}(A)}.$
			\item[(ii)] If $D=D_1\cup D_2,$ where $D_1\neq\emptyset, $ $D_2\neq\emptyset,$ then $D_1 \not \perp_A D_2.$
		\end{itemize}
        Moreover, any basis $D$ of $\overline{\mathcal{R}(A)}$ is a minimal $\mathcal K_A$-set if and only if $D$  satisfies the  property that: If $D=D_1\cup D_2,$ where $D_1\neq\emptyset, $ $D_2\neq\emptyset,$ then $D_1 \not \perp_A D_2.$
	\end{thm}

	\begin{proof}
		We first prove the sufficient part.  Let $T\in B_{A}(\mathbb{H})$ be such that $T$ preserves $A$-orthogonality at $D$.  Therefore, it follows from Proposition \ref{evt}  that each element of $D$ is an $A$-eigenvector of $T^\sharp T.$ It is easy to observe that  any two $A$-eigen spaces of $T^\sharp T$ with two distinct $A$-eigenvalues are $A$-orthogonal. Then it follows from (ii) that all the vectors of $D$ must belong to the same $A$-eigen space of $T^\sharp T.$ By (i),  it follows that all the vectors of $\overline{\mathcal{R}(A)}$ are $A$-eigenvectors of $T^\sharp T$ correspond to the same $A$-eigen value $\lambda$(say). Now, it follows from Proposition \ref{evt} and Theorem \ref{iso} that $T$ is a scalar multiple of an $A$-isometry. As $T \in B_{A}(\mathbb{H}) $ is chosen arbitrarily, $D$ is a $\mathcal K_A$-set for $ \mathbb{H}.$ 
		
	Next, we prove the necessary part. Let $D\subset S_{\mathbb{H}}$  be a  $\mathcal K_A$-set.  From Proposition \ref{dokn}, it follows that $\spn~ D=\overline{\mathcal{R}(A)}$ and so  $D$ satisfies (i). Next, we show that $D$ satisfies (ii). Suppose on the contrary that there exist nonempty $D_1,D_2\subset D$ with $D=D_1\cup D_2$, such that $D_1 \perp_A D_2.$
   Now, $$\mathbb{H}=\overline{\mathcal{R}(A)}\oplus \mathcal{N}(A)=(\spn~ D_1) \oplus \Big((\spn~ D_1)^{\perp_A}\cap \overline{\mathcal{R}(A)} \Big)\oplus \mathcal{N}(A).$$
Consider the operator $T \in B_{A}(\mathbb{H})$ defined by \[Tu=\begin{cases}
		u,&~\text{ if } u\in \spn~ D_1\\
		0,&~\text{ if } u\in \Big((\spn~ D_1)^{\perp_A}\cap \overline{\mathcal{R}(A)} \Big)\\
		0,&~\text{ if } u\in \mathcal{N}(A).
	\end{cases}\] 
	Clearly, $T$ is not a scalar multiple of an $A$-isometry.
	Now, $T(D_2)=\{0\}$ and so $T$ preserves $A$-orthogonality at each point of $D_2.$ Next, we show that $T$ preserves  $A$-orthogonality at each point of $D_1.$  Let $x\in D_1$ and $h\in \mathbb{X}$ such that $\langle x, h\rangle_A=0.$  Then there exist $h_1\in \spn~ D_1,$ $h_2\in(\spn~ D_1)^{\perp_A}\cap \overline{\mathcal{R}(A)}$ and $h_3\in \mathcal{N}(A)$ such that $h=h_1+h_2+h_3.$ Then 
	\begin{eqnarray*}
	\langle x, h\rangle_A=0&\implies &\langle x, h_1\rangle_A+ \langle x, h_2\rangle_A+\langle x, h_3\rangle_A=0\\
		&\implies &\langle x, h_1\rangle_A=0.
	\end{eqnarray*}
	Now,
		\begin{eqnarray*}
		\langle Tx, Th\rangle_A&=&\langle Tx, Th_1\rangle_A+ \langle Tx, Th_2\rangle_A+\langle Tx, Th_3\rangle_A\\
		&=&\langle x, h_1\rangle_A\\
		&=&0.
	\end{eqnarray*}
	Thus, $T$ preserves $A$-orthogonality at $x.$ Since $x$ is taken arbitrarily,  $T$ preserves  $A$-orthogonality at each point of $D_1$ and so  $T$ preserves  $A$-orthogonality at each point of $D.$ 
   This contradicts the fact that $D$ is a $\mathcal K_A$-set. Thus,  $D$ satisfies (ii).

	\end{proof}
    
We end this article with the following remark.
\begin{rem}
\begin{itemize}
    \item[(i)]The Blanco-Koldobsky-Turn\v{s}ek theorem states that  any bounded linear operator on a Banach space preserves Birkhoff-James orthogonality if and only if it is a scalar multiple of an isometry.   Theorem \ref{iso} extends the  Blanco-Koldobsky-Turn\v{s}ek theorem in the setting  of semi-Hilbert spaces.\\
    
    \item [(ii)] For finite-dimensional real Hilbert spaces, \cite[Th. 2.4]{SMP24} characterizes the  minimal $\mathcal{K}$-sets as a basis $D$  of the space satisfying  the property that if $D=D_1\cup D_2,$ where $D_1\neq\emptyset, $ $D_2\neq\emptyset,$ then $D_1 \not \perp D_2.$ Theorem \ref{k_set} extends the characterization of minimal $\mathcal{K}$-sets  in the setting of semi-Hilbert spaces.
\end{itemize}
\end{rem}

\subsection*{Acknowledgements}
Jayanta Manna would like to thank UGC, Govt. of India for the support
in the form of Senior Research Fellowship under the mentorship of Professor Kallol Paul and Dr Debmalya Sain.  Somdatta Barik would like to thank UGC, Govt. of India for the support
in the form of Senior Research Fellowship under the mentorship of Professor Kallol Paul.  	The research of  Professor Kallol Paul is supported by CRG Project bearing File no. CRG/2023/00716 of DST-SERB, Govt.
of India.
\subsection*{Declarations}

\begin{itemize}

	\item Conflict of interest
	
	The authors have no relevant financial or non-financial interests to disclose.
	
	\item Data availability 
	
	The manuscript has no associated data.
	
	\item Author contribution
	
	All authors contributed to the study. All authors read and approved the final version of the manuscript.
	
\end{itemize}


\begin{thebibliography}{1}

    \bibitem{ACG09} M. L. Arias, G. Corach and M. C. Gonzalez, \textit{Lifting properties in operator ranges.} Acta Sci. Math.
     (Szeged) \textbf{75} (2009), 635--653.

	\bibitem{ACG08} M. L. Arias, G. Corach and M. C. Gonzalez, \textit{Partial isometries in semi-Hilbertian spaces.} Linear
	Algebra Appl.  \textbf{428} (2008), 1460--1475.
	
     \bibitem{ACG_IEOT_08} M. L. Arias, G. Corach and M. C. Gonzalez, \textit{Metric properties of projections in semi-Hilbertian spaces.} Integr. Equ. Oper. Theory \textbf{62} (2008), 11--28.


     \bibitem{B35} G. Birkhoff,  \textit{Orthogonality in linear metric spaces.} Duke Math. J. \textbf{1} (1935), 169--172.
		
		\bibitem{BT06} A. Blanco and  A. Turn\v{s}ek, \textit{On maps that preserve orthogonality in normed spaces.} Proc. Roy. Soc. Edinburgh Sect. A \textbf{136} (2006), 709--716.

        \bibitem{D66} R. G. Douglas, On majorization, \textit{factorization and range inclusion of operators in Hilbert space.} Proc. Amer. Math. Soc. \textbf{17} (1966), 413--416.
        
      \bibitem{EN81} H. W. Engl and M.Z. Nashed, \textit{New extremal characterizations of generalized inverses of linear operators.} J. Math. Anal. Appl. \textbf{82} (1981), 566--586.

      \bibitem{J47} R. C. James, \textit{Orthogonality and linear functionals in normed linear spaces.} Trans.  Amer.  Math. Soc. \textbf{61} (1947), 265--292.
       \bibitem{K93} A. Koldobsky, \textit{Operators preserving orthogonality are isometries.} Proc. R. Soc. Edinburgh Sect. A 123 (1993), 835--837.

       
		\bibitem{MMPS25} J. Manna, K. Mandal, K. Paul and D. Sain, \textit{On directional preservation of orthogonality and its application to isometries.} Bull. Sci. Math.  \textbf{199} (2025), 103575.\\
		https://doi.org/10.1016/j.bulsci.2025.103575

        \bibitem{MPS24} A.  Mal, K. Paul and D. Sain, \textit{Birkhoff-James orthogonality and geometry of operator spaces},  Infosys Science Foundation Series, Springer Singapore, 2024. ISBN 978-981-99-7110-7, https://doi.org/10.1007/978-981-99-7111-4.

      \bibitem{P55} R. Penrose, \textit{A generalized inverse for matrices.} Math. Proc. Cambridge Philos. Soc. \textbf{51} (1955), 406--413.

      \bibitem{SMP24} D. Sain, J. Manna and K. Paul, \textit{ On local preservation of orthogonality and its application to isometries.} Linear Algebra Appl. \textbf{690} (2024), 112--131. 
      
     \bibitem{SBP24} A. Sen, R. Birbonshi and K. Paul, \textit{A note on the $A$-numerical range of semi-Hilbertian operators.} Linear Algebra Appl. \textbf{703} (2024), 268--288.
     
     \bibitem{SSP_AFA_21} J. Sen, D. Sain and K. Paul, \textit{Orthogonality and norm attainment of operators in semi-Hilbertian spaces.} Ann. Funct. Anal.  \textbf{12} (2021), 12 pp.

     \bibitem{SSP_BSM_21} J. Sen, D. Sain and K. Paul, \textit{On approximate orthogonality and symmetry of operators in semi-Hilbertian structure.}
     Bull. Sci. Math. \textbf{170} (2021), 22 pp.

      \bibitem{Z19} A. Zamani, \textit{Birkhof–James orthogonality of operators in semi-Hilbertian spaces and its applications.} Ann. Funct. Anal. \textbf{10} (2019), 433--445. 




\end{thebibliography}
\end{document}